\documentclass{amsart}%
\usepackage{amsmath}
\usepackage{amsfonts}
\usepackage{amssymb}
\usepackage{graphicx}%
\newtheorem{lemma}{Lemma}[section]
\newtheorem{theorem}[lemma]{Theorem}
\newtheorem{remark}[lemma]{Remark}

\newtheorem{coro}[lemma]{Corollary}
\newtheorem{definition}[lemma]{Definition}
\newtheorem{example}[lemma]{Example}

\title[Remotely Almost Periodic Solutions \ldots]{Remotely Almost Periodic Solutions of Scalar Differential
Equations}

\author{David Cheban}
\address[D. Cheban]{Moldova State University\\
Vladimir Andrunachievici Institue\\ of Mathematics and Informatics\\
Laboratory of Differential Equations\\
A. Mateevich Street 60\\
MD--2009 Chi\c{s}in\u{a}u, Moldova} \email[D.
Cheban]{david.ceban@usm.md, davidcheban@yahoo.com}

\date{\today}

\begin{document}

\begin{abstract}
The aim of this paper is to study the problem of existence of
remotely almost periodic solutions for the scalar differential
equation $x'=f(t,x),$ where $f:\mathbb R\times \mathbb R\to
\mathbb R$ is a continuous, monotone in $x$ and remotely almost
periodic in $t$ function. We prove that every solution $\varphi$ of this equation bounded on the semi-axis
$\mathbb R_{+}$ is remotely
almost periodic. This statement is a generalization of the
well-known Opial's theorem for remotely almost periodic scalar
differential equations. We also establish a similar statement for
scalar difference equations.
\end{abstract}

\maketitle

\section{Introduction}\label{Sec1}

The notion of remotely almost periodicity on the real axis
$\mathbb R$ for the scalar functions was introduced and studied by
Sarason \cite{Sar_1984}. Remotely almost periodic functions on the
semi-axis $\mathbb R_{+}$ with the values in the Banach space were
introduced and studied by Ruess and Summers \cite{RS_1986}.
Remotely almost periodic functions on the real axis with the
values in the Banach spaces were introduced and studied by
Baskakov \cite{Bas_2013}. He calls theses functions "almost
periodic at infinity". Remotely almost periodic on the real-axis
$\mathbb R$ solutions of ordinary differential equations with
remotely almost periodic on $\mathbb R$ coefficient were studied
by Maulen, Castillo, Kostic and Pinto \cite{MCKP_2021}.

Denote by $\mathbb R :=(-\infty,+\infty), \mathbb
R_{+}:=[0,+\infty), \mathbb T\in \{\mathbb R_{+}, \mathbb R\}$ and
$C(\mathbb T,\mathbb R)$ (respectively, $C(\mathbb T\times \mathbb
R,\mathbb R)$ the space of all continuous functions $\varphi
:\mathbb T\to \mathbb R$ (respectively, $f:\mathbb T\times \mathbb
R\to \mathbb R$) equipped with the compact-open topology. Let
$(C(\mathbb T,\mathbb R),\mathbb R,\sigma)$ (respectively,
$(C(\mathbb T\times \mathbb R,\mathbb R),\mathbb R,\sigma)$) be
the shift dynamical system \cite[Ch.I]{Che_2015} on the space
$C(\mathbb T,\mathbb R)$ (respectively, on the space $C(\mathbb
T\times \mathbb R,\mathbb R)$).

\begin{definition}\label{dfI1} Recall that a function $\varphi \in C(\mathbb R,\mathbb R)$
is said to be
\begin{enumerate}
\item almost periodic \cite[Ch.I]{Che_2024B} if for every
$\varepsilon
>0$ there exists a positive number $l=l(\varepsilon)$ such that on
every segment $[a,a+l]$ there exists at least one number $\tau$
for which $|\varphi(t+\tau)-\varphi(t)|<\varepsilon$ for all $t\in
\mathbb R$; \item asymptotically almost periodic
\cite[Ch.I]{Che_2024B} (respectively, asymptotically
$\tau$-periodic or asymptotically stationary) if there are $p,r\in
C(\mathbb R,\mathbb R)$ such that $\varphi =p+r$, $p$ is almost
periodic (respectively, $\tau$-periodic or stationary) and
$|r(t)|\to 0$ as $t\to +\infty$; \item remotely almost periodic
\cite{Che_2024P,Che_2024.1,RS_1986,Sar_1984} if for every
$\varepsilon
>0$ there exists a positive number $l=l(\varepsilon)$ such that on every segment
$[a,a+l]$ there exists at least one number $\tau$ and
$L(\varepsilon,\tau)>0$ so that
$|\varphi(t+\tau)-\varphi(t)|<\varepsilon$ for all $t\ge
L(\varepsilon,\tau)$.
\end{enumerate}
\end{definition}

\begin{remark}\label{remI1} 1. Every almost periodic function is asymptotically almost
periodic.

2. Every asymptotically almost periodic (respectively,
asymptotically $\tau$-periodic or asymptotically stationary)
function is remotely almost periodic (respectively, remotely
$\tau$-periodic or remotely stationary).

3. The converse statements (to the first two statements) are
false.
\end{remark}

\begin{definition}\label{defI2}
A function $\varphi \in C(\mathbb R,\mathbb R)$ is said to be:
\begin{enumerate}
\item remotely $\tau$-periodic if $\lim\limits_{t\to
+\infty}|\varphi(t+\tau)-\varphi(t)|=0$; \item remotely stationary
if it is remotely $\tau$-periodic for every $\tau \in \mathbb T$.
\end{enumerate}
\end{definition}

\begin{example}\label{exI1} {\em Consider a function $\varphi \in C(\mathbb R,\mathbb
R)$ defined by $\varphi(t):= \sin\ln (1+|t|)$ for all $t\in
\mathbb R$. Note that
\begin{equation}\label{eqI1}
|\varphi(t+\tau)-\varphi(t)|=2|\sin \frac{\ln
(1+|t+\tau|)-\ln(1+|t|)}{2}||\cos\frac{\ln
(1+|t+\tau|)+\ln(1+|t|)}{2}|\le \nonumber
\end{equation}
\begin{equation}\label{eqI1.1}
\ln \frac{\frac{1}{|t|}+|1+\frac{\tau}{|t|}|}{1+\frac{1}{|t|}}
\end{equation}
for all $t \in \mathbb R\setminus \{0\}$ and $\tau \in \mathbb R$.
Passing to the limit in the inequality (\ref{eqI1}) as $t\to
+\infty$ we obtain $\lim\limits_{t\to
+\infty}|\varphi(t+\tau)-\varphi(t)|=0$ for every $\tau \in
\mathbb R$ and, consequently, the function $\varphi$ is remotely
stationary.}
\end{example}

\begin{remark}\label{remI2} Every remotely $\tau$-periodic function is
remotely almost periodic. The converse statement is false.
\end{remark}

\begin{definition}\label{defI3}
A function $f\in C(\mathbb T\times \mathbb R,\mathbb R)$ is said
to be positively regular if for every $g\in H^{+}(f)$ and $v\in
\mathbb R$ there exists  a unique solution $\varphi(t,v,g)$ of the
equation $y'=g(t,v)$ passing through the point $v\in \mathbb R$ at
the initial moment $t=0$ and defined on $\mathbb
R_{+}:=[0,+\infty)$, where $H^{+}(f)$ is the closure of the set of
all translations $\{f^{h}|\ h\ge 0,\ f^{h}(t,x):=f(t+h,x)$ for all
$(t,x)\in \mathbb T\times \mathbb R\}$ in the space $C(\mathbb
T\times \mathbb R,\mathbb R)$.
\end{definition}

It is well known the following Massera's result \cite{Mas_1950}
(see also \cite[Ch.XII]{Fin_1974} and \cite[Ch.II]{Plis_1964}).

\begin{definition}\label{defAPP1}
A function $f\in C(\mathbb T\times \mathbb R,\mathbb R)$ is said
to be asymptotically almost periodic (respectively, asymptotically
$\tau$-periodic or asymptotically stationary) in time $t$
uniformly with respect to (shortly w.r.t.) $x$ on every compact
subset from $\mathbb R$ if there are functions $p,r\in C(\mathbb
R\times \mathbb R,\mathbb R)$ such that
\begin{enumerate}
\item $f(t,x)=p(t,x)+r(t,x)$ for all $(t,x)\in \mathbb T\times
\mathbb R$; \item the function $p$ is almost periodic
(respectively, $\tau$-periodic or stationary) in time $t$ and
$\lim\limits_{t\to +\infty}|r(t,x)|=0$ uniformly w.r.t. $x$ on
every compact subset from $\mathbb R$.
\end{enumerate}
\end{definition}

\begin{theorem}\label{thI1} (\emph{Massera}). Assume that the following
conditions are fulfilled:
\begin{enumerate}
\item the function $f\in C(\mathbb T\times \mathbb R,\mathbb R)$
is $\tau$-periodic ($\tau >0$) in time, i.e., $f(t+\tau,x)=f(t,x)$
for every $(t,x)\in \mathbb T\times \mathbb R$; \item $f$ is
positively regular; \item the equation
\begin{equation}\label{eq01}
x'(t)=f(t,x)
\end{equation}
admits a solution $\varphi(t,u_0,f)$ bounded on $\mathbb R_{+}$.
\end{enumerate}

Then the solution $\varphi(t,u_0,f)$ is asymptotically
$\tau$-periodic, i.e., there exists a $\tau$-periodic function
$p:\mathbb R\to \mathbb R$ such that $\lim\limits_{t\to
+\infty}|\varphi(t,u_0,f)-p(t)|=0.$
\end{theorem}

For the asymptotically $\tau$-periodic equations (\ref{eq01}) we
have the following result \cite{Che_2023wp}.

\begin{theorem}\label{thI2} (\emph{Cheban}) Suppose that the following
conditions hold:
\begin{enumerate}
\item the function $f\in C(\mathbb R\times \mathbb R,\mathbb R)$
is asymptotically $\tau$-periodic ($\tau >0$) in time; \item $f$
is positively regular; \item the equation (\ref{eq01}) admits a
solution $\varphi(t,u_0,f)$ bounded on $\mathbb R_{+}$.
\end{enumerate}

Then the solution $\varphi(t,u_0,f)$ is remotely $\tau$-periodic,
i.e., $\lim\limits_{t\to
+\infty}|\varphi(t+\tau,u_0,f)-\varphi(t,u_0,f)|=0.$
\end{theorem}

The following natural question arises.

\textbf{Problem}. Is the theorem \ref{thI2} true for remotely
almost periodic scalar differential equations (\ref{eq01}).

In general, the answer to this question is negative (Example
\ref{exAAP1}). The main result of this paper in the following
theorem is contained.

\begin{theorem}\label{thI3} Suppose that the following conditions
hold:
\begin{enumerate}
\item the function $f\in C(\mathbb T\times \mathbb R,\mathbb R)$
is asymptotically almost periodic (respectively, asymptotically
$\tau$-periodic or asymptotically stationary) in time; \item $f$
is positively regular; \item the function $f$ is monotone with
respect to $x$ uniformly in $t\in \mathbb R$, i.e., for every
$x_1,x_2\in \mathbb R$ with $x_1\le x_2$ we have $f(t,x_1)\le
f(t,x_2)$ for all $t\in \mathbb T$; \item the equation
(\ref{eq01}) admits a solution
$\varphi(t,u_0,f)$ bounded on $\mathbb R_{+}$.
\end{enumerate}

Then the solution $\varphi(t,u_0,f)$ is remotely almost periodic
(respectively, remotely $\tau$-periodic or remotely stationary).
\end{theorem}

The proof of this statement in Subsection \ref{Sec4.1} is given
(Theorem \ref{thM3.1}).

The aim of this paper is studying the remotely almost periodic
solutions of scalar differential equations. This study continues
the author's series of works devoted to the study of remotely
almost periodic motions of dynamical systems and solutions of
differential equations
\cite{Che_2009},\cite{Che_2023wp}-\cite{Che_2024.1}.

The paper is organized as follows. In the second section we
collect some known notions and facts about remotely almost
periodic motions of dynamical systems and remotely almost periodic
functions. In the third section we study the remotely almost
periodic motions of the one-dimensional monotone nonautonomous
dynamical systems. The fourth section is dedicated to the
application of the general results obtained in the third section
to the scalar differential and difference equations.

\section{Preliminary}\label{Sec2}

Let $X$ and $Y$ be two complete metric spaces, let $\mathbb Z
:=\{0,\pm 1, \pm 2, \ldots \}$, $\mathbb S =\mathbb R$ or $\mathbb
Z$,  $\mathbb S_{+}=\{t \in \mathbb S |\quad t \ge 0 \}$ and
$\mathbb S_{-}=\{t \in \mathbb S| \quad t \le 0 \}$. Let $\mathbb
T \in \{\mathbb S,\ \mathbb S_{+}\}$ and $(X,\mathbb S_{+},\pi)$
(respectively, $(Y,\mathbb S, \sigma )$) be an autonomous
one-sided (respectively, two-sided) dynamical system on $X$
(respectively, on $Y$).

Let $(X,\mathbb T,\pi)$ be a dynamical system.

\begin{definition}\label{defSP1} A point $x\in X$ (respectively, a motion $\pi(t,x)$) is
said to be:
\begin{enumerate}
\item[-] stationary, if $\pi(t,x)=x$ for all $t\in \mathbb T$;
\item[-] $\tau$-periodic ($\tau >0$ and $\tau \in \mathbb T$), if
$\pi(\tau,x)=x$; \item[-] asymptotically stationary (respectively,
asymptotically $\tau$-periodic), if there exists a stationary
(respectively, $\tau$-periodic) point $p\in X$ such that
\begin{equation}\label{eqAP1*}
\lim\limits_{t\to \infty}\rho(\pi(t,x),\pi(t,p))=0.\nonumber
\end{equation}
\end{enumerate}
\end{definition}

\begin{theorem}\label{thAAP1}\cite[Ch.I]{Che_2009} A point $x\in X$ is asymptotically $\tau$-periodic if and
only if the sequences $\{\pi(k\tau,x)\}_{k=0}^{\infty}$ converges.
\end{theorem}

\begin{definition}\label{defLS1} A point $\widetilde{x}\in X$ is said
to be $\omega$-limit for $x\in X$ if there exists a sequence
$\{t_k\}\subset \mathbb S_{+}$ such that $t_k\to +\infty$ and
$\pi(t_k,x)\to \widetilde{x}$ as $k\to \infty$.
\end{definition}

Denote by $\omega_{x}$ the set of all $\omega$-limit points of
$x\in X$.

If $(X,\mathbb S,\pi)$ is a two-sided dynamical system, then
$\alpha_{x}:=\{p\in X:$ there exists a sequence $t_k\to -\infty$
such that $\pi(t_k,x)\to p$ as $k\to \infty\}$.

\begin{definition}\label{defSAP1}  We will call
a point $x\in X$ (respectively, a motion $\pi(t,x)$) remotely
$\tau$-periodic ($\tau\in \mathbb T$ and $\tau
>0$) if
\begin{equation}\label{eqSAP_1}
 \lim\limits_{t\to
+\infty}\rho(\pi(t+\tau,x),\pi(t,x))=0 .
\end{equation}
\end{definition}

\begin{remark}\label{remS1.0} The motions of dynamical systems possessing the property
(\ref{eqSAP_1}) in the works of Cryszka \cite{Gry_2018} and
Pelczar \cite{Pel_1985} were studied.
\end{remark}

\begin{definition}\label{defLS02} A point $x$ is called Lagrange
stable (respectively, positively Lagrange stable), if the
trajectory $\Sigma_{x}:=\{\pi(t,x)|\ t\in \mathbb T\}$
(respectively, semi-trajectory $\Sigma^{+}_{x}:=\{\pi(t,x)|\ t\ge
0\}$) is a precompact subset of $X$.
\end{definition}

\begin{theorem}\label{th1.3.9}\cite[Ch.I]{Che_2020} Let $x\in X$ be
positively Lagrange stable and $\tau\in\mathbb T\ (\tau >0)$. Then
the following statements are equivalent:
\begin{enumerate}
\item[a.] the motion $\pi(t,x)$ is remotely $\tau$-periodic;
\item[b.] every point $p\in\omega_{x}$ is $\tau$-periodic.
\end{enumerate}
\end{theorem}

\begin{definition}\label{defSAP2} A point $x$ (respectively, a
motion $\pi(t,x)$) is said to be remotely stationary, if it is
remotely $\tau$-periodic for every $\tau \in \mathbb T$.
\end{definition}

\begin{coro}\label{corSAP1} Let $x\in X$ be positively Lagrange stable. Then the following
statements are equivalent:
\begin{enumerate}
\item[a.] the motion $\pi(t,x)$ is remotely stationary; \item[b.]
every point $p\in\omega_{x}$ is stationary.
\end{enumerate}
\end{coro}
\begin{proof} This statement follows directly from Definition \ref{defSAP2}
and Theorem \ref{th1.3.9}.
\end{proof}

\begin{definition}\label{defRAP0} A subset $A\subseteq \mathbb
T$ is said to be relatively dense in $\mathbb T$ if there exists a
positive number $l\in \mathbb T$ such that $[a,a+l]\bigcap A\not=
\emptyset$ for all $a\in \mathbb T$, where $[a,a+l]:=\{x\in
\mathbb T|\ a\le x\le a+l\}$.
\end{definition}

\begin{remark}\label{remDT1} For every $\tau >0$ ($\tau \in \mathbb
T$) the set $A:=\{k\tau |\ k\in \mathbb Z\}\bigcap \mathbb T$ is
relatively dense in $\mathbb T$.
\end{remark}

\begin{definition}\label{defAP1} A point $x\in X$ of the dynamical
system $(X,\mathbb T,\pi)$ is said to be:
\begin{enumerate}
\item almost periodic if for every $\varepsilon >0$ the set
\begin{equation}\label{eqAP1}
\mathcal P(\varepsilon,p):=\{\tau \in \mathbb T|\
\rho(\pi(t+\tau,p),\pi(t,p))<\varepsilon \ \ \mbox{for all}\ t\in
\mathbb T\}\nonumber
\end{equation}
is relatively dense in $\mathbb T$; \item positively Poisson
stable if $x\in \omega_{x}$; \item asymptotically stationary
(respectively, asymptotically $\tau$-periodic, asymptotically
almost periodic or positively asymptotically Poisson stable) if
there exists a stationary (respectively, $\tau$-periodic, almost
periodic or positively Poisson stable) point $p\in X$ such that
\begin{equation}\label{eqAP3}
\lim\limits_{t\to \infty}\rho(\pi(t,x),\pi(t,p))=0.\nonumber
\end{equation}
\end{enumerate}
\end{definition}

\begin{definition}\label{defRAP1} A point $x\in X$ (respectively,
a motion $\pi(t,x)$) is said to be remotely almost periodic
\cite{RS_1986} if for arbitrary positive number $\varepsilon$
there exists a relatively dense subset $\mathcal
P(\varepsilon,x)\subseteq \mathbb T$ such that for every $\tau \in
\mathcal P(\varepsilon,x)$ there exists a number
$L(\varepsilon,x,\tau)>0$ for which we have
\begin{equation}\label{eqRAP1}
\rho(\pi(t+\tau,x),\pi(t,x))<\varepsilon \nonumber
\end{equation}
for all $t\ge L(\varepsilon,x,\tau)$.
\end{definition}

\begin{remark}\label{remAP1} Every almost periodic
(respectively< asymptotically almost periodic) point $x\in X$ is
remotely almost periodic.
\end{remark}

\begin{lemma}\label{lRAP_01} \cite{Che_2024_1} Every remotely $\tau$-periodic (respectively,
remotely stationary) point $x$ of the dynamical system $(X,\mathbb
T,\pi)$ is remotely almost periodic.
\end{lemma}

\begin{lemma}\label{lRAP_010} A point $x$
is remotely $\tau$-periodic (respectively, remotely stationary) if
and only if for every $\varepsilon
>0$ there exists a relatively dense in $\mathbb T$ subset $\mathcal
P(x,\varepsilon)$ such that
\begin{enumerate}
\item $\{\tau \mathbb Z\}\bigcap \mathbb T \subset \mathcal
P(x,\varepsilon)$ (respectively, $\mathbb T \subseteq \mathcal
P(x,\varepsilon)$) and \item for every $\tau \in \mathcal
P(x,\varepsilon)$ there exists a number $L(x,\varepsilon,\tau)>0$
for which we have
\begin{equation}\label{eqE1}
\rho(\pi(t+\tau,x),\pi(t,x))<\varepsilon \nonumber
\end{equation}
for all $t\ge L(x,\varepsilon,\tau)$.
\end{enumerate}
\end{lemma}
\begin{proof} This statement follows directly from the corresponding
definitions.
\end{proof}

\begin{definition}\label{defEAP1} A positively invariant subset $M\subseteq
X$is said to be equi-almost periodic if for every $\varepsilon >0$
the set
\begin{equation}\label{eqEAP1}
\mathcal F(\varepsilon,M):=\{\tau \in \mathbb R |\
\rho(\pi(\tau,x),x)<\varepsilon\ \ \ \forall\ x\in M\} \nonumber
\end{equation}
is relatively dense.
\end{definition}

\begin{theorem}\label{th1SRAP} \cite{Che_2024_1,RS_1986} Assume that the point $x\in X$ is
positively Lagrange stable. The following conditions are
equivalent:
\begin{enumerate}
\item the point $x$ is remotely almost periodic; \item the set
$\omega_{x}$ is equi-almost periodic.
\end{enumerate}
\end{theorem}

\section{One-dimensional Monotone Nonautonomous Dynamical
Systems}\label{Sec3}

Let $Y$ be a complete metric space, $(Y,\mathbb T, \sigma )$ be an
autonomous dynamical system on $Y$.

\begin{definition}\label{def2.1}(\emph{One-dimensional cocycle with the base
$(Y ,\mathbb T ,\sigma)$}). The triplet $\langle \mathbb R,
\varphi ,(Y , \mathbb T, \sigma )\rangle$ (or briefly $\varphi$)
is said to be a cocycle (see, for example,
\cite[Ch.II,IX]{Che_2015} and \cite{Sel_1971}) on the state space
$\mathbb R$ with the base $(Y,\mathbb T,\sigma )$ if the mapping
$\varphi : \mathbb R_{+} \times Y \times \mathbb R \to \mathbb R $
satisfies the following conditions:
\begin{enumerate}
\item $\varphi (0,y,u)=u $ for all $u\in \mathbb R$ and $y\in Y$;
\item $\varphi (t+\tau ,y ,u)=\varphi (t,\varphi (\tau
,u,y),\sigma(\tau,y))$ for all $ t, \tau \in \mathbb R_{+},u \in
\mathbb R$ and $ y \in Y$; \item the mapping $\varphi $ is
continuous.
\end{enumerate}
\end{definition}

\begin{definition}\label{defC1} (\emph{Monotone cocycles}). A cocycle $ \langle \mathbb R ,
\varphi,(Y,\mathbb T,\sigma)\rangle $ is said to be monotone if
for all $u_1,u_2\in \mathbb R$ with $u_1\le u_2$ we have
$\varphi(t,u_1,y)\le \varphi(t,u_2,y)$ for every $t\in \mathbb
R_{+}$ and $y\in Y$.
\end{definition}

\begin{definition}\label{def2.2} (\emph{Skew-product dynamical system}). Let $ \langle \mathbb R,
\varphi,(Y,\mathbb T,\sigma)\rangle $ be a cocycle on $ \mathbb R,
X:=\mathbb R\times Y$ and $\pi $ be a mapping from $ \mathbb R_{+}
\times X $ to $X$ defined by the equality $\pi =(\varphi
,\sigma)$, i.e., $\pi (t,(u,y))=(\varphi (t,u,y),\sigma(t,y))$ for
all $ t\in \mathbb R_{+}$ and $(u,y)\in \mathbb R\times Y$. The
triplet $ (X,\mathbb R_{+}, \pi)$ is an autonomous dynamical
system and it is called \cite{Sel_1971} a skew-product dynamical
system.
\end{definition}

\begin{definition}\label{def2.3} (\emph{Nonautonomous dynamical system}.) Let
$ \mathbb T_{1}\subseteq \mathbb T _{2} $ be two sub-semigroups of
the group $\mathbb S, (X,\mathbb T _{1},\pi ) $ and $(Y ,\mathbb T
_{2}, \sigma )$ be two autonomous dynamical systems and $h: X \to
Y$ be a homomorphism from $(X,\mathbb T _{1},\pi )$ to $(Y ,
\mathbb T _{2},\sigma)$ (i.e., $h(\pi(t,x))=\sigma(t,h(x)) $ for
all $t\in \mathbb T _{1} $, $ x \in X $ and $h $ is continuous),
then the triplet $\langle (X,\mathbb T _{1},\pi ),$ $ (Y ,$
$\mathbb T _{2},$ $ \sigma ),h \rangle $ is called (see
\cite{Bro79} and \cite{Che_2015}) a nonautonomous dynamical
system.
\end{definition}

\begin{example} \label{ex2.4} (A nonautonomous dynamical system generated by the cocycle $\varphi$)
{\em Let $\langle \mathbb R, \varphi ,(Y ,\mathbb T,\sigma)\rangle
$ be a cocycle, $(X,\mathbb T_{+},\pi ) $ be a skew-product
dynamical system ($X=\mathbb R\times Y, \pi =(\varphi,\sigma)$)
and $h= pr _{2}: X \to Y ,$ then the triplet $\langle (X,\mathbb
T_{+},\pi ),$ $(Y ,\mathbb T,\sigma ),h \rangle $ is a
nonautonomous dynamical system.}
\end{example}

\begin{definition}\label{defIS011} A function $\gamma :Y\to \mathbb R$ is
said to be upper semi-continuous\index{upper/low semi-continuous
function} (respectively, lower semi-continuous)
\cite[Ch.III]{Loj_1982} at the point $y_0$  if
$$
\lim\limits_{\rho(y,y_0)\to 0}\gamma(y)\le \gamma(y_0)
$$
(respectively,
$$
\lim\limits_{\rho(y,y_0)\to 0}\gamma(y)\ge \gamma(y_0)).
$$
The function $\gamma$ is called upper (respectively, lower)
semi-continuous on $A\subseteq Y$ if it is upper (respectively,
lower) semi-continuous at every point $y_0\in A$.
\end{definition}

\begin{definition}\label{defAB1}
Let $M\subseteq X$ be a compact subset of $X$ such that $h(M)=Y$.
Denote by $M_{y}:=h^{-1}(y)=\{x\in M|\ h(x)=y\}$,
$I_{y}:=pr_{1}(M_{y})\subseteq W$,
\begin{equation}\label{eqIS0.01}
\alpha_{M}(y):=\inf\{u\in I_{y}\}\nonumber
\end{equation}
and
\begin{equation}\label{eq002}
\beta_{M}(y):=\sup\{u\in I_{y}\}.\nonumber
\end{equation}
\end{definition}

\begin{lemma}\label{lIS1}\cite[Ch.V]{Che_2024B} Let $M\subseteq X$ be a compact subset of
$X$ such that $h(M)=Y$. Then the following statements hold:
\begin{enumerate}
\item for every $y\in Y$ the set $M_{y}:=h^{-1}(y)\bigcap X$ is a
nonempty compact subset from $X_{y}:=h^{-1}(y)$; \item
$(\alpha(y),y),(\beta(y),y)\in M_{y}=I_{y}\times \{y\}$;

\item the map
$$
\alpha :Y\to \mathbb R,\ y\to \alpha(y)
$$
(respectively,
$$
\beta :Y\to \mathbb R,\ y\to \beta(y))
$$
is lower semi-continuous (respectively, upper semi-continuous).
\end{enumerate}
\end{lemma}

\begin{remark}\label{remIS1} In what follows, in the notation $\alpha_{M}$ (respectively, $\beta_{M}$),
we will omit the index $M$ if this does not lead to a
misunderstanding.
\end{remark}

\begin{lemma}\label{lIS2}\cite[Ch.V]{Che_2024B} Let $M$ be a
compact subset of $X$ such that $h(M)=Y$,
$\widehat{M}_{y}:=[\alpha(y),\beta(y)]\times \{y\}$ and
$\widehat{M} =\bigcup\{ \widehat{M}_{y}|\ y\in Y\}$. Then the
following statement hold:
\begin{enumerate}
\item $M_{y}\subseteq \widehat{M}_{y}$ for every $y\in Y$; \item
$M\subseteq \widehat{M}$; \item the set $\widehat{M}$ is a compact
subset of $X$; \item $\alpha_{\widehat{M}}=\alpha_{M}$ and
$\beta_{\widehat{M}}=\beta_{M}$.
\end{enumerate}
\end{lemma}

\begin{definition}\label{def2.4} The cocycle $\varphi$ is called
$V-$monotone \cite[Ch.XII]{Che_2015} if there exists a continuous
function $\mathcal V : \mathbb R \times \mathbb R \times Y \to
\mathbb R _{+}$ with the following properties:
\begin{enumerate}
\item $ \mathcal V (u_1,u_2,\omega)\ge 0 $ for all $ y \in Y $ and
$ u_1,u_2 \in \mathbb R;$ \item $ \mathcal V (u_1,u_2,y)=0 $ if
and only if $u_1=u_2;$ \item $ \mathcal V (\varphi
(t,u_1,y),\varphi (t,u_2,y), \sigma(t,y)) \le \mathcal V
(u_1,u_2,y) $ for all $u_1,u_2 \in \mathbb R, y \in Y $ and $ t\in
\mathbb T _{+}.$
\end{enumerate}
\end{definition}

\begin{theorem}\label{thV1} \cite[Ch.XII]{Che_2015}
Let $Y$ be a compact metric space, $ \langle \mathbb R ,
\varphi,(Y,\mathbb T,\sigma)\rangle $ be a $V$-monotone cocycle
and $M\subseteq X:=\mathbb R\times Y$ be a compact positively
invariant subset of the skew-product dynamical system $(X,\mathbb
T_{+},\pi)$. Then the cocycle $\varphi$ is positively uniformly
stable on $M$, i.e., for arbitrary positive number $\varepsilon$
there exists a positive number $\delta =\delta(\varepsilon,M)>0$
such that $|u_1-u_2|<\delta$ ($(u_1,y),(u_2,y)\in M$) implies
$|\varphi(t,u_1,y)-\varphi(t,u_2,y)|<\varepsilon$ for all
$(t,y)\in \mathbb T_{+}\times Y$.
\end{theorem}

\begin{definition}\label{defD1} A compact invariant set $M\subseteq
X$ of two sided nonautonomous dynamical system $\langle (X,\mathbb
S,\pi),(Y,\mathbb S,\sigma),h\rangle$ is called distal in the
negative direction (negatively distal) if
\begin{equation}\label{eqD_1}
\inf\limits_{t\le 0}\rho(\pi(t,x_1),\pi(t,x_2))>0\nonumber
\end{equation}
for all $x_1,x_2\in M$ with $x_1\not= x_2$ and $h(x_1)=h(x_2)$.
\end{definition}

\begin{lemma}\label{lV1} \cite[Ch.XII]{Che_2015} Let $M\subseteq X$
be a compact invariant set of $(X,\mathbb S_{+},\pi)$. If the
nonautonomous dynamical system  is uniformly stable, then $\langle
(X,\mathbb S_{+},\pi),(Y,$ $\mathbb S,$ $\sigma),$ $h\rangle$ is
distal in the negative direction.
\end{lemma}

\begin{theorem}\label{thV2} \cite[Ch.IX]{Che_2015} Let $ \langle \mathbb R ,
\varphi,(Y,\mathbb S,\sigma)\rangle $ be a cocycle with the
following properties:
\begin{enumerate}
\item the skew-product dynamical system $(X,\mathbb S_{+},\pi)$,
generated by the cocycle $\varphi$, admits a compact invariant set
$M\subset X$; \item the cocycle $\varphi$ is positively uniformly
stable on $M$.
\end{enumerate}

Then the skew-product dynamical system $(X,\mathbb S_{+},\pi)$
generates on $M$ a two-sided dynamical system $(M,\mathbb S,\pi)$.
\end{theorem}

\begin{theorem}\label{thV3} \cite[Ch.XII,p.370]{Che_2015} Let $ \langle \mathbb R ,
\varphi,(Y,\mathbb S,\sigma)\rangle $ be a cocycle, $M\subseteq
X:=\mathbb R\times Y$ be a compact positively invariant subset of
the skew-product dynamical system $(X,\mathbb S_{+},\pi)$ and the
following conditions are fulfilled:
\begin{enumerate}
\item $Y$ is compact and minimal (this means that every trajectory
from $Y$ is dense in $Y$, i.e., $Y=H(y)$ for every $y\in Y$, where
$H(y):=\overline{\{\sigma(t,y)|\ t\in \mathbb T\}}$); \item
$|\varphi(t,u_1,y)-\varphi(t,u_2,y)|\le |u_1-u_2|$ for all
$(u_1,y),(u_2,y)\in M$ and $t\in \mathbb S_{+}$; \item the set
$M_{y}:=\{x\in M:\ h(x)=y\}$ is convex for every $y\in Y$.
\end{enumerate}

Then there exists a continuous mapping $\gamma =(\nu,Id_{Y}):Y\to
M$ with the properties: $pr_{2}(\gamma(y))=y$ and
$\gamma(\sigma(t,y))=\pi(t,\gamma(y))$ for all $(t,y)\in \mathbb
S\times Y$ (or equivalently: $\nu :Y\to pr_{1}(M)$ and
$\nu(\varphi(t,u,y))=\varphi(t,\nu(y),y)$ for every $(t,y)\in
\mathbb S\times Y$).
\end{theorem}

Let $\mathfrak N_{y}:=\{\{t_n\}\subset \mathbb T|$ such that
$\sigma(t_n,y)\to y$ as $n\to \infty$ and $\mathfrak N_{y}^{\pm
\infty}:=\{\{t_n\}\in \mathfrak N_{y}|$ such that $t_n\to \pm
\infty$ as $n\to \infty\}$.

\begin{lemma}\label{lV2.1}\cite[ChIII]{Che_2020} Suppose that the following conditions
are fulfilled:
\begin{enumerate}
\item $y\in Y$  is a two-sided Poisson stable point, i.e., $x\in
\omega_{x}=\alpha_{x}$; \item $\langle (X,\mathbb S,\pi),$
$(Y,\mathbb S,$ $\sigma),$ $h\rangle $ is a two-sided
nonautonomous dynamical system; \item $X$ is a compact space;
\item
\begin{equation}\label{eqP1**}
\inf \limits _{t\le 0}\rho (\pi(t,x_{1}),\pi(t,x_{2}))>0\nonumber
\end{equation}
for all $x_{1},x_{2} \in X_{y}\quad (x_{1}\not= x_{2})$.
\end{enumerate}

Then for every pair of points $x_1,x_2\in X_{y}$ with $x_1\not=
x_2$ there are the sequences $\{t^{-}_{n}\}\in\mathfrak
N^{-\infty}_{y}$ and $\{t^{+}_{n}\}\in\mathfrak N^{+\infty}_{y}$
such that
\begin{equation}\label{eqP_1.1}
\lim\limits_{n\to \infty}\pi(t_{n}^{\pm},x_{i})=x_{i}\ (i=1,2)
.\nonumber
\end{equation}
\end{lemma}

\begin{theorem}\label{thV4.1} Let $ \langle \mathbb R ,
\varphi,(Y,\mathbb S,\sigma)\rangle $ be a cocycle, $M\subseteq
X:=\mathbb R\times Y$ be a compact and invariant subset of the
skew-product dynamical system $(X,\mathbb S_{+},\pi)$ and the
following conditions are fulfilled:
\begin{enumerate}
\item the set $Y$ is compact and minimal; \item
$|\varphi(t,u_1,y)-\varphi(t,u_2,y)|\le |u_1-u_2|$ for all
$(u_1,y),(u_2,y)\in X$ and $t\in \mathbb S_{+}$.
\end{enumerate}

Then for every $y\in Y$ and $x_{i}=(u_{i},y)\in M_{y}:=M\bigcap
X_{y}$ ($i=1,2$) there exists a constant $C=C(v_1,v_2,y)\in
\mathbb R$ such that
\begin{equation}\label{eqV1}
\varphi(t,v_1,y)-\varphi(t,v_2,y)=C\nonumber
\end{equation}
for all $t\in \mathbb S$.
\end{theorem}
\begin{proof}
By Theorem \ref{thV2} the dynamical system $(X,\mathbb S_{+},\pi)$
generates on $M$ a two-sided dynamical system $(M,\mathbb S,\pi)$
and, consequently, a two-sided nonautonomous dynamical system
\begin{equation}\label{eqNDS_1}
\langle (M,\mathbb S, \pi),(Y,\mathbb S,\sigma),h \rangle
\end{equation}
is defined.

By Theorem \ref{thV1} and Lemma \ref{lV1} the nonautonomous
dynamical system (\ref{eqNDS_1}) is negatively distal and by Lemma
\ref{lV2.1} for all $y\in \omega_{y_0}$ and
$x_1=(v_1,y),x_{2}=(v_2,y)\in M_{y}$ there are sequences
$\{t_{k}^{\pm}\}\in \mathfrak N_{y}^{\pm \infty}$ such that
\begin{equation}\label{eqV2}
\pi(t_{k}^{+\infty},x_{i})\to x_{i}\ \ \mbox{and}\  \
\pi(t_{k}^{-\infty},x_{i})\to x_{i}\nonumber
\end{equation}
as $k\to \infty$ ($i=1,2$).

Consider the function $\varphi :\mathbb S\to \mathbb R_{+}$
defined by the equality
\begin{equation}\label{eqV3}
\varphi(t):=\rho(\pi(t,x_1),\pi(t,x_2))=|\varphi(t,v_1,y)-\varphi(t,v_2,y)|
\nonumber
\end{equation}
for every $t\in \mathbb S$.

Note that $\varphi$ possesses the following properties:
\begin{enumerate}
\item the function $\varphi$ is bounded on $\mathbb S$;  \item
$\varphi$ is monotone decreasing, i.e., $t_1\le t_2$ implies
$\varphi(t_2)\le \varphi(t_1)$ for all $t_1,t_2\in \mathbb S$;
\item there exist the limits $\lim\limits_{t\to \pm
\infty}\varphi(t)=c_{\pm}$, $c_{\pm}\in \mathbb R_{+}$; \item
$c_{-}\le \varphi(t)\le c_{+}$ for all $t\in \mathbb S$; \item
$c_{-}=c_{+}=\varphi(0)$ and, consequently,
$\varphi(t)=\varphi(0)$ for every $t\in \mathbb S$.
\end{enumerate}

The first statement follows from the fact that the set $M$ is
compact and invariant.

To prove the second statement we note that if $t_2\ge t_1$, then
$$
\varphi(t_2)=|\varphi(t_2,v_1,y)-\varphi(t_2,v_2,y)|=
$$
$$
|\varphi(t_2-t_1,\varphi(t_1v_1,y),\sigma(t_1,y))-\varphi(t_2-t_1,\varphi(t_1,v_2,y),\sigma(t_1,y))|\le
$$
$$
|\varphi(t_1,v_1,y)-\varphi(t_1,v_2,y)|=\varphi(t_1) .
$$
The third statement follows directly from the first and second
statements. The fourth statement is obvious.

Finally, we will show that $c_{-}=c_{+}$. To this end we note that
\begin{equation}\label{eqV5}
c_{-}=\lim\limits_{k\to
\infty}|\varphi(t_{k}^{-},v_1,y)-\varphi(t_{k}^{-},v_2,y)|=|v_1-v_2|=\varphi(0)
\end{equation}
and
\begin{equation}\label{eqV6}
c_{+}=\lim\limits_{k\to
\infty}|\varphi(t_{k}^{+},v_1,y)-\varphi(t_{k}^{+},v_2,y)|=|v_1-v_2|=\varphi(0)
.
\end{equation}
From (\ref{eqV5})-(\ref{eqV6}) we obtain $c_{-}=\varphi(0)=c_{+}$
and, consequently, $\varphi(t)=\varphi(0)$ for all $t\in \mathbb
S$. Theorem is completely proved.
\end{proof}

\begin{theorem}\label{thV5} Let $ \langle \mathbb R ,
\varphi,(Y,\mathbb S,\sigma)\rangle $ be a cocycle, $x_0\in
X:=\mathbb R\times Y$ be a positively Lagrange stable point of the
skew-product dynamical system $(X,\mathbb S_{+},\pi)$ and the
following conditions are fulfilled:
\begin{enumerate}
\item the point $y_0:=h(x_0)$ is positively Lagrange stable; \item
the $\omega$-limit set $\omega_{y_0}$ of $y_0$ is minimal; \item
$|\varphi(t,u_1,y)-\varphi(t,u_2,y)|\le |u_1-u_2|$ for all
$(u_1,y),(u_2,y)\in X$ and $t\in \mathbb S_{+}$.
\end{enumerate}

Then the following statements hold:
\begin{enumerate}
\item there exists a continuous mapping $\gamma =(\nu,Id_{Y}):Y\to
X:=\mathbb R \times Y$ such that
\begin{equation}\label{eqV_1}
\gamma (\sigma(t,y))=\pi(t,\gamma(y))
\end{equation}
or equivalently $\nu :Y\to \mathbb R$
\begin{equation}\label{eqV_2}
\nu(\sigma(t,y))=\varphi(t,\nu(y),y)
\end{equation}
for all $(t,y)\in \mathbb S\times Y$; \item for every $y\in
\omega_{y_0}$ and $x=(v,y)\in \omega_{x_0}\bigcap X_{y}$ there
exists a constant $C=C(x,y)\in \mathbb R$ such that
\begin{equation}\label{eqV_3}
\varphi(t,v,y)=\nu(\sigma(t,y))+C\nonumber
\end{equation}
for all $t\in \mathbb S$.
\end{enumerate}
\end{theorem}
\begin{proof}
Since the points $x_0$ and $y_0$ are positively Lagrange stable
then the sets $\omega_{x_0}$ and $\omega_{y_0}$ are nonempty,
compact, invariant and $h(\omega_{x_0})=\omega_{y_0}$. By Theorem
\ref{thV2} the dynamical system $(X,\mathbb S_{+},\pi)$ generates
on $\omega_{x_0}$ a two-sided dynamical system
$(\omega_{x_0},\mathbb S,\pi)$ and, consequently, it is defined a
two-sided nonautonomous dynamical system
\begin{equation}\label{eqNDS_01}
\langle (\omega_{x_0},\mathbb S, \pi),(\omega_{y_0},\mathbb
S,\sigma),h \rangle .
\end{equation}

Denote by $A:=\omega_{x_0}$ and $M:=\widehat{A}=\bigcup
\{\widehat{A}_{y}:\ y\in Y\}$, where
$\widehat{A}_{y}:=[\alpha(y),\beta(y)]\times \{y\}$,
$$
\alpha(y):=\sup\{v:\ (v,y)\in \omega_{x_0}\bigcap X_{y}\}
$$
and
$$
\beta(y):=\inf\{v:\ (v,y)\in \omega_{x_0}\bigcap X_{y}\}.
$$
By Lemma \ref{lIS2} the set $M$ is compact, invariant,
$\omega_{x_0}=A\subseteq \widehat{A}=M$ and, consequently,
$h(M)=\omega_{y_0}$ because the set $\omega_{y_0}$ is minimal. By
Theorem \ref{thV2} the dynamical system $(X,\mathbb S_{+},\pi)$
generates on the set $M$ a two-sided dynamical system $(M,\mathbb
S,\pi)$. Consider the nonautonomous dynamical system
\begin{equation}\label{eqNDS2}
\langle (M,\mathbb S,\pi),(\omega_{y_0},\mathbb S,\sigma),h\rangle
.\nonumber
\end{equation}

By Theorem \ref{thV1} and Lemma \ref{lV1} the nonautonomous
dynamical system (\ref{eqNDS_01}) is negatively distal and by
Theorem \ref{thV4.1} for all $y\in\omega_{y_0}$ and
$(v_1,y),(v_2,y)\in M_{y}$ there exists a constant
$C=C(v_1,v_2,y)\in \mathbb R$ such that
\begin{equation}\label{eqV2.1}
\varphi(t,v_1,y)-\varphi(t,v_2,y)=C
\end{equation}
for all $t\in \mathbb S$.

On the other hand by Theorem \ref{thV4.1} there exists a
continuous mapping $\gamma =(\nu,Id_{Y}):Y\to X:=\mathbb R \times
Y$ such that (\ref{eqV_1}) (or equivalently (\ref{eqV_2})) holds
for all $(t,y)\in \mathbb S\times Y$. Consider now the points
$x_1=(v,y),x_2:=(\nu(y),y)\in M_{y}$ and the motions
$\pi(t,x_1)=(\varphi(t,v,y),\sigma(t,y))$ and
$$
\pi(t,x_2)=(\varphi(t,\nu(y),y),\sigma(t,y))=(\nu(\sigma(t,y),\sigma(t,y))
$$
for every $t\in \mathbb S$. By (\ref{eqV2.1}) we have
\begin{equation}\label{eqV2.2}
\varphi(t,v,y)-\nu(\sigma(t,y))=C\nonumber
\end{equation}
for all $t\in \mathbb S$. Theorem is proved.
\end{proof}

\begin{theorem}\label{thV6} Let $ \langle \mathbb R ,
\varphi,(Y,\mathbb S,\sigma)\rangle $ be a cocycle, $x_0\in
X:=\mathbb R\times Y$ be a positively Lagrange stable point of the
skew-product dynamical system $(X,\mathbb S_{+},\pi)$ and the
following conditions are fulfilled:
\begin{enumerate}
\item the point $y_0:=h(x_0)$ is positively Lagrange stable and
remotely stationary (respectively, remotely $\tau$-periodic or
remotely almost periodic); \item the $\omega$-limit set
$\omega_{y_0}$ of $y_0$ is minimal; \item
$|\varphi(t,u_1,y)-\varphi(t,u_2,y)|\le |u_1-u_2|$ for all
$(u_1,y),(u_2,y)\in X$ and $t\in \mathbb S_{+}$.
\end{enumerate}

Then the point $x_0$ is remotely almost periodic.
\end{theorem}
\begin{proof}
To prove this theorem it suffices to show that the $\omega$-limit
set $\omega_{x_0}$ of $x_0$ is equi-almost periodic (see Theorem
\ref{th1SRAP}).

By Theorem \ref{thV5} there exists a continuous mapping $\gamma
=(\nu,Id_{Y}):Y\to X:=\mathbb R \times Y$ such that:
\begin{enumerate}
\item
\begin{equation}\label{eqV_1.1}
\gamma (\sigma(t,y))=\pi(t,\gamma(y))
\end{equation}
(or equivalently $\nu :Y\to \mathbb R$
\begin{equation}\label{eqV_2.1}
\nu(\sigma(t,y))=\varphi(t,\nu(y),y))
\end{equation}
for all $(t,y)\in \mathbb S\times Y$ and \item for every $y\in
\omega_{y_0}$ and $x=(v,y)\in \omega_{x_0}\bigcap X_{y}$ there
exists a constant $C=C(x,y)\in \mathbb R$ such that
\begin{equation}\label{eqV_3.1}
\varphi(t,v,y)=\nu(\sigma(t,y))+C \nonumber
\end{equation}
for all $t\in \mathbb S$.
\end{enumerate}

Let $\varepsilon$ be an arbitrary positive number. Since the set
$\omega_{y_{0}}$ is compact and $\gamma :\omega_{y_0}\to
X:=\mathbb R\times \omega_{y_0}$ is continuous then there exists a
positive number $\delta =\delta(\varepsilon)$ such that
\begin{equation}\label{eqMRAP_01}
d(y_1,y_2)<\delta\ \ \mbox{implies}\ \ \
\rho(\gamma(y_1),\gamma(y_2))<\varepsilon.
\end{equation}
Now we will show that the $\omega$-limit set $\omega_{x_0}$ of
$x_0$ is equi-almost periodic.

To prove this fact we fix an arbitrary positive number
$\varepsilon$. Since the point $y_0$ is positively Lagrange stable
and remotely almost periodic then by Theorem \ref{th1SRAP} the set
$\omega_{y_0}$ is equi-almost periodic. This means that the set
\begin{equation}\label{eqMRAP_2}
\mathcal F(\varepsilon,\omega_{y_0}):=\{\tau \in \mathbb S|\
d(\sigma(\tau,y),y)<\delta(\varepsilon)\ \forall y\in
\omega_{y_0}\}
\end{equation}
is relatively dense. According to the choice of
$\delta(\varepsilon)$ we have
\begin{equation}\label{eqMRAP_3}
d(\gamma (\sigma(\tau,y)),\gamma(y))<\varepsilon
\end{equation}
for all $y\in \omega_{y_0}$ and $\tau \in \mathcal
F(\varepsilon,\omega_{y_0})$.

Since $\gamma(\sigma(t,y))=\pi(t,\gamma(y))$ for every $(t,y)\in
\mathbb S\times \omega_{y_0}$ and taking into account
(\ref{eqMRAP_01}) and (\ref{eqMRAP_2}) we obtain
\begin{eqnarray}\label{eqMRAP_4}
& \rho(\sigma(t+\tau,\gamma(y)),\sigma(t,\gamma(y)))
=d(\gamma(\sigma(t+\tau,y)),\gamma(\sigma(t,y)))<\varepsilon
\end{eqnarray}
for all $t\in \mathbb S$ and $y\in \omega_{y_0}$.

Let now $p=(v,q)\in \omega_{x_0}$ be an arbitrary point then
\begin{equation}\label{eqP_1}
\pi(t,p)=(\varphi(t,v,q),\sigma(t,q)) \nonumber
\end{equation}
for every $t\in \mathbb S$.

We fix $\tau \in \mathcal F(\varepsilon,\omega_{y_0})$ then taking
into account (\ref{eqV_1.1})-(\ref{eqV_2.1}) and
(\ref{eqMRAP_2})-(\ref{eqMRAP_4}) we obtain
$$
d(\pi(t+\tau,p),\pi(t,p))=d(\pi(t+\tau,(v,q)),\pi(t,(v,q)))=
$$
$$
d((\varphi(t+\tau,v,q),\sigma(t+\tau,q)),(\varphi(t,v,q),\sigma(t,q)))=
$$
$$
|\varphi(t+\tau,v,q)-\varphi(t,v,q)|+d(\sigma(t+\tau,q),\sigma(t,q))=
$$
$$
|\nu(\sigma(t+\tau,q))+C(v,q)-\nu(\sigma(t,q))-C(v,q)|+d(\sigma(t+\tau,q),\sigma(t,q))=
$$
$$
|\nu(\sigma(t+\tau,q))-\nu(\sigma(t,q))|+d(\sigma(t+\tau,q),\sigma(t,q))=
$$
$$
d(\gamma(\sigma(t+\tau,q)),\gamma(\sigma(t,q)))<\varepsilon
$$
for all $t\in \mathbb S$, i.e., $\mathcal
F(\varepsilon,\omega_{y_0})\subseteq \mathcal
F(\varepsilon,\omega_{x_0})$ and, consequently, $\mathcal
F(\varepsilon,\omega_{x_0})$ is relatively dense.

This means that the set $\omega_{x_0}$ is equi-almost periodic. To
finish the proof we note that if the set $\omega_{y_0}$ consists
of a single periodic  (respectively, stationary) trajectory, then
the set $\omega_{x_0}$ is so. Thus the set $\omega_{x_0}$ is
equi-almost periodic. Since the point $x_0$ is positively Lagrange
stable then by Theorem \ref{th1SRAP} the point $x_0$ is remotely
almost periodic. Theorem is completely proved.
\end{proof}

\section{Applications}\label{Sec4}

\subsection{Scalar Differential Equations}\label{Sec4.1}

\begin{example}\label{exS1} (Scalar differential equations) {\em In
this Example we suppose that $\mathbb T\in \{\mathbb R_{+},\
\mathbb R\}$. Consider a differential equation
\begin{equation}\label{eqS1}
u'=f(\sigma(t,y),u),\ \ (y\in Y)
\end{equation}
where $f\in C(Y\times \mathbb R,\mathbb R)$.

\begin{definition}\label{defS0} A function $f\in C(Y\times \mathbb
R,\mathbb R)$ is said to be regular \cite[Ch.IV]{Sel_1971} if for
every $u\in\mathbb R$ and $y\in Y$ the equation (\ref{eqS1}) has a
unique solution $\varphi(t,u,y)$ passing through the point $u\in
\mathbb R$ at the initial moment $t=0$ and defined on $\mathbb
R_{+}$.
\end{definition}

Everywhere in this subsection we suppose that the right hand side
$f$ of the equation (\ref{eqS1}) is regular.

From the general properties of solutions for the equation
(\ref{eqS1}) we have
\begin{enumerate}
\item $\varphi(0,u,y)=u$ for every $u\in \mathbb R$ and $y\in Y$;
\item
$\varphi(t+\tau,u,y)=\varphi(t,\varphi(\tau,u,y),\sigma(\tau,y))$
for all $t,\tau\in \mathbb R_{+}$, $u\in \mathbb R$ and $y\in Y$;
\item the mapping $(t,u,y)\to \varphi(t,u,y)$ from $\mathbb
R_{+}\times \mathbb R\times Y$ to $\mathbb R$ is continuous; \item
$u_1\le u_2$ implies $\varphi(t,u_1,y)\le \varphi(t,u_2,y)$ for
all $u_1,u_2\in \mathbb R$, $t\in\mathbb R_{+}$ and $y\in Y$.
\end{enumerate}
Thus every equation (\ref{eqS1}) with the regular right hand side
$f$ generates a monotone cocycle $\langle \mathbb R,\varphi,
(Y,\mathbb T,\sigma)\rangle$ with continuous time $\mathbb
R_{+}$.}
\end{example}

\begin{example}\label{exS2}
{\em Consider the equation
\begin{equation}\label{eqM1}
u'=f(t,u),
\end{equation}
where $f\in C(\mathbb R\times \mathbb R,\mathbb R)$. Along with
the equation (\ref{eqM1}) we consider the family of equations
\begin{equation}\label{eqM2}
u'=g(t,u),
\end{equation}
where $g\in H(f):=\overline{\{f^{\tau}|\ \tau\in\mathbb R\}}$.
Suppose that the function $f$ is regular \cite{Sel_1971}, i.e.,
for every $g\in H(f)$ and $u\in\mathbb R$ there exists a unique
solution $\varphi(t,u,g)$ of the equation (\ref{eqM2}) defined on
$\mathbb R_{+}$. Denote by $Y=H(f)$ and $(Y,\mathbb R, \sigma)$
the shift dynamical system on $Y$ induced by the Bebutov's
dynamical system $(C(\mathbb R\times \mathbb R,\mathbb R),\mathbb
R,\sigma)$. Now the family of the equations (\ref{eqM2}) can be
written as (\ref{eqS1}). Namely,
\begin{equation}\label{eqM2.1}
u'=F(\sigma(t,g),u)\ \ (g\in H(f))\nonumber
\end{equation}
if we take $Y=H(f)$ and the mapping $F\in C(Y\times \mathbb
R,\mathbb R)$ defined by $F(g,u):=g(0,u)$, for every $g\in H(f)$
and $u\in\mathbb R$ because
\begin{equation}\label{eqM2.2}
F(\sigma(t,g),u)=F(g^{t},u)=g^{t}(0,u)=g(t,u) \nonumber
\end{equation}
for all $(t,g)\in \mathbb R\times H(f)$. Thus the equation
(\ref{eqM1}) with the regular right hand side $f$ generates an
one-dimensional monotone cocycle $\langle R,\varphi,(Y,\mathbb
R,\sigma)\rangle$.}
\end{example}

\begin{definition}\label{defMF1} A function $f\in C(\mathbb R\times \mathbb R,\mathbb
R)$ is said to be monotone (order-preserving) w.r.t. $x\in \mathbb
R$ uniformly w.r.t. $t\in \mathbb R$ if $x_1\le x_2$ implies
$f(t,x_1)\le f(t,x_2)$ for all $(t,x_i)\in \mathbb R\times \mathbb
R$ ($i=1,2$).
\end{definition}

\begin{remark}\label{remMF1} Assume that the function $f\in C(\mathbb R\times \mathbb R,\mathbb
R)$ is monotone w.r.t. $x\in \mathbb R$ uniformly w.r.t. $t\in
\mathbb R$. Then every function $g\in H(f)$ possesses with the
same property.
\end{remark}

\begin{lemma}\label{lD1.1} \cite{Che_2021},\cite[Ch.IV]{Che_2024B}
Let the function $f\in C(\mathbb R\times \mathbb R,\mathbb R)$ be
regular, monotone and $\varphi$ be the cocycle generated by the
equation (\ref{eqM1}).

Then we have
\begin{equation}\label{eqM3}
|\varphi(t,u_1,g)-\varphi(t,u_2,g)|\le |u_1-u_2|
\end{equation}
for all $(t,u_i,g)\in \mathbb R_{+}\times \mathbb R\times H(f)$
($i=1,2$) and, consequently, the cocycle $\varphi$ is positively
uniformly stable.
\end{lemma}

\begin{theorem}\label{thM3.1} Suppose that the following assumptions
are fulfilled:
\begin{enumerate}
\item[-] the function $f\in C(\mathbb R\times \mathbb R,\mathbb
R)$ is positively Lagrange stable; \item[-] $f$ is remotely almost
periodic (respectively, remotely $\tau$-periodic or remotely
stationary) in $t\in\mathbb R$ uniformly w.r.t. $u$ on every
compact subset from $\mathbb R$; \item[-] $\omega_{f}$ is a
minimal set of the shift dynamical system $(C(\mathbb R,\mathbb
R),\mathbb R,\sigma))$; \item[-] the function $f\in C(\mathbb
R\times \mathbb R,\mathbb R)$ is monotone and regular; \item[-]
the equation (\ref{eqM1}) admits a solution $\varphi(t,u_0,f)$ bounded on $\mathbb R_{+}$.
\end{enumerate}

Then the solution $\varphi(t,u_0,f)$ is remotely almost periodic
(respectively, remotely $\tau$-periodic or remotely stationary).
\end{theorem}
\begin{proof}
Let $f\in C(\mathbb R\times \mathbb R,\mathbb R)$ and $(C(\mathbb
R\times \mathbb R,\mathbb R),\mathbb R,\sigma)$ be the shift
dynamical system on $C(\mathbb R\times \mathbb R,\mathbb R)$.
Denote by $Y:=H(f)$ and $(Y,\mathbb R,\sigma)$ the shift dynamical
system on $H(f)$ induced by $(C(\mathbb R\times \mathbb R,\mathbb
R),\mathbb R,\sigma)$. Consider the cocycle $\langle\mathbb
R,\varphi,(Y,\mathbb R,\sigma)\rangle$ generated by (\ref{eqM1})
(see Example \ref{exS2}). Since the function $f$ is regular and
monotone then by Lemma \ref{lD1.1} the cocycle $\varphi$ generated
by (\ref{eqM1}) satisfies the condition (\ref{eqM3}) and,
consequently, the cocycle $\varphi$ is positively uniformly
stable. Denote by $x_0:=(u_{0},f)\in X=\mathbb R\times H(f)$. It
easy to check that under the conditions of Theorem the point $x_0$
of the skew-product dynamical system $(X,\mathbb R_{+},\pi)$
($X:=\mathbb R\times H(f)$ and $\pi :=(\varphi,\sigma)$) is
positively Lagrange stable. Thus the $\omega$-limit set
$\omega_{x_0}$ of the point $x_0$ is a nonempty, compact and
invariant set of the dynamical system $(X,\mathbb R_{+},\pi)$. Let
\begin{equation}\label{eqNDSODE1}
\langle (X,\mathbb R_{+},\pi),(Y,\mathbb R,\sigma),h\rangle
\end{equation}
be a nonautonomous dynamical system associated by the cocycle
$\varphi$ ($X:=\mathbb R\times H(f)$, $\pi :=(\varphi,\sigma)$ and
$h:=pr_{2}:X\to Y:=H(f)$). Note that $h(\omega_{x_0})=\omega_{f}$
and by Theorem \ref{thV2} it is well defined the two-sided
nonautonomous dynamical system $\langle (\omega_{x_0},\mathbb
R,\pi),(\omega_{f},\mathbb R,\sigma),h\rangle$. Since the function
$f$ is remotely almost periodic and positively Lagrange stable
then the set $\omega_{f}$ is a compact minimal set consisting of
almost periodic functions $g\in \omega_{f}$. To finish the proof
of Theorem it suffices to apply Theorem \ref{thV6} to
nonautonomous dynamical system (\ref{eqNDSODE1}). Theorem is
proved.
\end{proof}

\begin{remark}\label{remDE1} Note that the monotony requirement of
the right-hand side $f$ of the equation (\ref{eqM1}) w.r.t. the
spatial variable is essential. Below we will give an example
confirming the above.
\end{remark}

\begin{example}\label{exDE3} (Opial's example) {\em Consider the scalar almost periodic differential equation
\begin{equation}\label{eqO1}
x'=f(t,x)
\end{equation}
with non-monotone (w.r.t. $x\in \mathbb R$) right-hand side $f$
which has all the solutions bounded on $\mathbb R$, but does not
admit an almost periodic solution \cite{Op_1961} (see also
\cite[Ch.II]{Che_2015} and \cite{Zhi72}).

\begin{lemma}\label{lO1} Let $f\in C(\mathbb R\times \mathbb R,\mathbb R)$
be an almost periodic in time $t$ uniformly w.r.t. $x$ on every
compact subset $Q$ from $\mathbb R$. If the equation (\ref{eqO1})
admits a solution bounded on $\mathbb R$ and remotely almost periodic, then it also admits at least one almost periodic
solution.
\end{lemma}
\begin{proof} Let $\varphi\in C(\mathbb R,\mathbb R)$ be a
bounded and remotely almost periodic solution of the equation
(\ref{eqO1}). Since the function $\varphi$ is bounded on $\mathbb
R$ then the function $\varphi$ is Lagrange stable in the shift
dynamical system $(C(\mathbb R,\mathbb R),\mathbb R,\sigma)$ and,
consequently, its $\omega$-limit set $\omega_{\varphi}$ is a
nonempty, compact and invariant set. By Theorem \ref{th1SRAP} the
set $\omega_{\varphi}$ is equi-almost periodic and, consequently,
it consists of almost periodic functions.

Since the function $f\in C(\mathbb R\times \mathbb R,\mathbb R)$
is almost periodic in time $t$ uniformly w.r.t. $x$ on every
compact subset $Q$ from $\mathbb R$, then there exists a sequence
$\{h_{k}\}\subset \mathbb R$ such that $f(t+h_k,x)\to f(t,x)$
uniformly w.r.t. $(t,x)\in \mathbb R\times Q$ as $k\to \infty$ for
every compact subset $Q\subset \mathbb R$. On the other hand the
function $\varphi$ is positively Lagrange stable and,
consequently, without loss of generality we can suppose that the
sequence $\{\varphi^{h_k}\}$ converges. Denote its limit by $\psi$
then it is easy to check that the function $\psi\in
\omega_{\varphi}$ and, consequently, it is almost periodic.

On the other hand the function $\varphi^{h_k}$ is a solution of the equation
\begin{equation}\label{eqO2}
x'=f^{h_k}(t,x)=f(t+h_k,x)\nonumber
\end{equation}
 bounded on the
$\mathbb R$. Since $f^{h_k}\to f$ (in the space $C(\mathbb R\times \mathbb
R,\mathbb R)$) and $\varphi^{h_k}\to \psi$ (in the space
$C(\mathbb R,\mathbb R)$) as $k\to \infty$, then $\psi \in C(\mathbb
R,\mathbb R)$ is an almost periodic solution of the "limiting"
equation, i.e., of the equation (\ref{eqO1}). Lemma is proved.
\end{proof}

\begin{coro}\label{corO1} The Opial's example does not admit a
remotely almost periodic solution.
\end{coro}
\begin{proof} This statement directly follows from Lemma
\ref{lO1}. Indeed. If we assume that it is not true, then by Lemma
\ref{lO1} the equation (\ref{eqO1}), figuring in Opial's example,
has at least one almost periodic solution. The last statement
contradicts to Opial's result.
\end{proof}
}
\end{example}

\begin{remark}\label{remODE1} 1. If the function $f\in C(\mathbb R\times \mathbb R,\mathbb
R)$ is asymptotically almost periodic in $t$ uniformly w.r.t. $x$
on every compact subset $Q$ from $\mathbb R$, then
\begin{enumerate}
\item $f$ is remotely almost periodic in $t$ uniformly w.r.t. $x$
on every compact subset $Q$ from $\mathbb R$; \item $f$ is
positively Lagrange stable; \item the $\omega$-limit set
$\omega_{f}$ of $f$ is minimal.
\end{enumerate}

2. A function $\phi \in C(\mathbb R,\mathbb R)$, defined by
$\phi(t):=\sin (t+\ln(1+|t|))$ for every $t\in \mathbb R$,
possesses \cite{Che_2024.1} the following properties:
\begin{enumerate}
\item $\phi$ is positively Lagrange stable; \item the
$\omega$-limit set $\omega_{\phi}$ of $\phi$ is minimal; \item the
function $\phi$ is not asymptotically almost periodic.
\end{enumerate}
\end{remark}

\begin{coro}\label{corAAP1} Under the conditions of Theorem
\ref{thM3.1} if the right hand side $f$ is asymptotically almost
periodic in time $t$ uniformly w.r.t. $x$ on every compact subset
from $\mathbb R$, then every solution
of the equation (\ref{eqM1}) bounded on $\mathbb R_{+}$ is remotely almost periodic.
\end{coro}
\begin{proof}
This statement directly follows from  Theorem \ref{thM3.1} because
every asymptotically almost periodic in time $t$ uniformly w.r.t.
$x$ on every compact subset $Q$ from $\mathbb R$ function $f\in
C(\mathbb R\times R,\mathbb R)$ possesses the following
properties:
\begin{enumerate}
\item $f$ is remotely almost periodic in time $t$ uniformly w.r.t.
$x$ on every compact subset $Q$ from $\mathbb R$; \item $f$ is
positively Lagrange stable; \item the $\omega$-limit set
$\omega_{f}$ of $f$ is a minimal set in the shift dynamical system
$(C(\mathbb R\times \mathbb R,\mathbb R),\mathbb R,\sigma)$.
\end{enumerate}
\end{proof}

In the relation with Corollary \ref{corAAP1} the following natural
question arise:

\textbf{Question.} Is there Opial's theorem for asymptotically
almost periodic equation?

Unfortunately, the answer to this question, in general, is
negative. Below we give a simple example confirming above.

\begin{example}\label{exAAP1} {\em
Consider the differential equation (\ref{eqM1}) with the right
hand side $f\in C(\mathbb R\times \mathbb R, \mathbb R)$ defined
by
\begin{equation}\label{eqCE1}
f(t,x):=\frac{2t\cos (t^2+\pi^{3})^{1/3}}{3(t^{2}+\pi^{3})^{2/3}}
\end{equation}
for all $(t,x)\in \mathbb R\times \mathbb R$. It is easy to check
that the function $f$ defined by (\ref{eqCE1}) satisfies all
conditions of Theorem \ref{thM3.1} and it is asymptotically almost
periodic (in fact asymptotically stationary) in time uniformly
w.r.t. $x\in \mathbb R$ because
\begin{equation}\label{eqCE2}
\sup\limits_{x\in \mathbb R}|f(t,x)|=\frac{2t|\cos
(t^2+\pi^{3})^{1/3}|}{3(t^{2}+\pi^{3})^{2/3}}\le
\frac{2t}{3(t^{2}+\pi^{3})^{2/3}}\to 0 \nonumber
\end{equation}
as $t\to +\infty$.

Let $\varphi(t,x_0,f)$ be a solution of the equation
\begin{equation}\label{eqCE3}
x'=\frac{2t\cos
(t^2+\pi^{3})^{1/3}}{3(t^{2}+\pi^{3})^{2/3}}.\nonumber
\end{equation}
It easy to see that
\begin{equation}\label{eqCE4}
\varphi(t,x_0,f)=x_0 +\sin (t^2+\pi^{3})^{1/3} \nonumber
\end{equation}
and
\begin{equation}\label{eqCE5}
|\varphi(t+\tau,x_0,f)-\varphi(t,x_0,f)|=|\sin(\pi^{3}+(t+\tau)^{2})^{1/3}-\sin(\pi^{3}+t^{2})^{1/3}|=
\nonumber
\end{equation}
$$
2|\sin\frac{(\pi^{3}+(t+\tau)^{2})^{1/3}-(\pi^{3}+t^{2})^{1/3}}{2}
\cos\frac{(\pi^{3}+(t+\tau)^{2})^{1/3}+(\pi^{3}+t^{2})^{1/3}}{2}|=
$$
$$
2|\sin\frac{\tau (t+\tau)}{2((\pi^{3}+(t+\tau)^{2})^{2/3}+
(\pi^{3}+(t+\tau)^{2})^{1/3}(\pi^{3}+t^{2})^{1/3}
+(\pi^{3}+t^{2})^{2/3})}|\cdot
$$
$$
|\cos(\frac{(\pi^{3}+(t+\tau)^{2})^{1/3}+(\pi^{3}+t^{2})^{1/3}}{2}|\le
$$
$$
\frac{|\tau (t+\tau)|}{(\pi^{3}+(t+\tau)^{2})^{2/3}+
(\pi^{3}+(t+\tau)^{2})^{1/3}(\pi^{3}+t^{2})^{1/3}
+(\pi^{3}+t^{2})^{2/3}} \to 0
$$
as $t\to +\infty$ for every fixed $\tau \in \mathbb R$, i.e., the
solution $\varphi(t,x_0,f)$ is remotely stationary.

Note that the solution $\varphi (t,x_0,f)$ is not asymptotically
stationary. To establish this fact it is sufficient to note that
\begin{equation}\label{eqCE6}
\varphi(t^{1}_{k},x_0,f)=x_0 \nonumber
\end{equation}
and
\begin{equation}\label{eqCE7}
\varphi(t^{2}_{k},x_0,f)=1+x_0, \nonumber
\end{equation}
where $t_{k}^{1}=(-\pi^{3}+(k\pi)^{3})^{1/2}$ and
$t_{k}^{2}=(-\pi^{3}+(\pi/2+2k\pi)^{3})^{1/2}$ for all $k\in
\mathbb N$.}
\end{example}

\subsection{Scalar Difference Equations}\label{Sec4.2}

\begin{example}\label{exDE1}
{\em Consider a difference equation
\begin{equation}\label{eqSDDE1}
u(t+1)=f(t,u(t)),
\end{equation}
where $f\in C(\mathbb Z_{+}\times \mathbb R,\mathbb R)$. Along
with the equation (\ref{eqSDDE1}) we will consider its
$H^{+}$-class, i.e., the family of equations
\begin{equation}\label{eqSDDE2}
v(t+1)=g(t,v(t)), \ (g\in H^{+}(f))
\end{equation}
where $H^{+}(f):=\overline{\{f^{\tau}|\ \tau\in\mathbb Z_{+}\}}$.

Denote by $\varphi(t,v,g)$ the solution of the equation
(\ref{eqSDDE2}) with the initial condition $\varphi(0,v,g)=v.$
From the general properties of difference equations it follows
that:
\begin{enumerate}
\item $\varphi(0,v,g)=v$ for all $v\in \mathbb R$ and $g\in
H^{+}(f);$ \item $\varphi (t+\tau,v,g)=\varphi (t,\varphi
(\tau,v,g),\sigma(\tau,g))$ for all $t,\tau \in  \mathbb Z_{+}$
and $(v,g)\in \mathbb R\times H^{+}(f)$; \item the mapping
$\varphi:\mathbb Z_{+}\times \mathbb R \times H^{+}(f)\to \mathbb
R$ is continuous.
\end{enumerate}

Thus every equation (\ref{eqSDDE1}) generates a cocycle $\langle
\mathbb R,$ $\varphi,$ $ (H^{+}(f), \mathbb Z,\sigma)\rangle$ over
$(H^{+}(f), \mathbb Z,\sigma)$ with the fibre $\mathbb R$.}
\end{example}

\begin{lemma}\label{l8.02}\cite{Che_2023.1}, \cite[Ch.IV]{Che_2024B}
Let $f\in C(\mathbb Z_{+}\times \mathbb R,\mathbb R).$ Suppose
that the function $f$ is monotone non-decreasing w.r.t. variable
$u\in \mathbb R,$ i.e., $u_1\le u_2$ implies $f(t,u_1)\le
f(t,u_2)$ for all $t\in\mathbb Z_{+}$ and $u_i\in \mathbb R$
($i=1,2$).

Then $\varphi(t,v_1,g) \le \varphi(t,v_2,g)$ for all $t\in Z_{+}$,
$v_1,v_2\in\mathbb R$ with $v_1\le v_2$ and $g\in H^{+}(f)$.
\end{lemma}

Everywhere below in this subsection we will suppose that the right
hand side $f\in C(\mathbb Z_{+} \times \mathbb R,\mathbb R)$ is
monotone non-decreasing w.r.t. variable $u\in \mathbb R,$ i.e.,
$f(t,u_1)\le f(t,u_2)$ for every $t\in\mathbb Z_{+}$ if $u_1\le
u_2$.

\begin{lemma}\label{l8.03} Assume that the function $f\in C(\mathbb Z_{+}\times \mathbb R,\mathbb
R)$ satisfies the condition
\begin{equation}\label{eqC1}
|f(t,x_1)-f(t,x_2)|\le |x_1-x_2|
\end{equation}
for all $(t,x_{i})\in \mathbb Z_{+}\times \mathbb R$ ($i=1,2$).

Then the following statements hold:
\begin{enumerate}
\item
\begin{equation}\label{eqC2}
|g(t,y_1)-g(t,y_2)|\le |y_1-y_2|
\end{equation}
for every $g\in H^{+}(f)$, $y_1,y_2\in \mathbb R$ and $t\in
\mathbb Z_{+}$; \item
\begin{equation}\label{eqC3}
|\varphi(t,v_1,g)-\varphi(t,v_2,g)|\le |v_1-v_2| \nonumber
\end{equation}
for all $t\in \mathbb Z_{+}$, $v_1,v_2\in \mathbb R$ and $g\in
H^{+}(f)$.
\end{enumerate}
\end{lemma}
\begin{proof}
If $g\in H^{+}(f)$, then there exists a sequence $\{h_{k}\}\subset
\mathbb Z_{+}$ such that $f^{h_k}\to g$ in the space $C(\mathbb
Z_{+}\times \mathbb R,\mathbb R)$ as $k\to \infty$. According to
(\ref{eqC1}) we have
\begin{equation}\label{eqC4}
|f(t+h_k,x_1)-f(t+h_k,x_2)|\le |x_1-x_2|
\end{equation}
for every $(t,x)\in \mathbb Z_{+}\times \mathbb R$ and $k\in
\mathbb N$. Passing to the limit in (\ref{eqC4}) as $k\to \infty$
we obtain (\ref{eqC2}).

To prove the second statement we note that by (\ref{eqC2}) we
receive
\begin{equation}\label{eqC5}
|\varphi(t+1,v_1,g)-\varphi(t+1,v_2,g)|=
\end{equation}
$$
|g(t,\varphi(t,v_1,g))-g(t,\varphi(t,v_2,g))|\le
|\varphi(t,v_1,g)-\varphi(t,v_2,g)|
$$
for all $t\in \mathbb Z_{+}$, $v_1,v_2\in \mathbb R$ and $g\in
H^{+}(f)$. From (\ref{eqC5}) we obtain
$$
|\varphi(t,v_1,g)-\varphi(t,v_2,g)|\le
|\varphi(t-1,v_1,g)-\varphi(t-1,v_2,g)| \le
$$
$$
\ldots \le |\varphi(1,v_1,g)-\varphi(1,v_2,g)| \le |v_1-v_2|
$$
for every $t\in \mathbb N$, $v_1,v_2\in \mathbb R$ and $g\in
H^{+}(f)$. Lemma is completed proved.
\end{proof}

\begin{theorem}\label{thDEM3.1} Suppose that the following assumptions
are fulfilled:
\begin{enumerate}
\item[-] the function $f\in C(\mathbb Z_{+}\times \mathbb
R,\mathbb R)$ is positively Lagrange stable; \item[-] $f$ is
remotely almost periodic (respectively, remotely $\tau$-periodic
or remotely stationary) in $t\in\mathbb R$ uniformly w.r.t. $u$ on
every compact subset from $\mathbb R$; \item[-] $\omega_{f}$ is a
minimal set of the shift dynamical system $(C(\mathbb
Z_{+},\mathbb R),\mathbb Z_{+},\sigma))$; \item[-] the function
$f\in C(\mathbb Z_{+}\times \mathbb R,\mathbb R)$ is monotone and
\begin{equation}\label{eqL1}
|f(t,x_1)-f(t,x_2)|\le |x_1-x_2| \nonumber
\end{equation}
for every $t\in \mathbb Z_{+}$ and $x_1,x_2\in \mathbb R$;
\item[-] the equation (\ref{eqSDDE1}) admits a solution $\varphi(t,u_0,f)$ bounded on $\mathbb
R_{+}$.
\end{enumerate}

Then the solution $\varphi(t,u_0,f)$ is remotely almost periodic
(respectively, remotely $\tau$-periodic or remotely stationary).
\end{theorem}
\begin{proof}
Let $f\in C(\mathbb Z_{+}\times \mathbb R,\mathbb R)$ and
$(C(\mathbb Z_{+}\times \mathbb R_{+},\mathbb R),\mathbb
Z_{+},\sigma)$ be the shift dynamical system on $C(\mathbb
Z_{+}\times \mathbb R,\mathbb R)$. Denote by $Y:=H^{+}(f)$ and
$(Y,\mathbb Z_{+},\sigma)$ the shift dynamical system on
$H^{+}(f)$ induced by $(C(\mathbb Z_{+}\times \mathbb R,\mathbb
R),\mathbb Z_{+},\sigma)$. Consider the cocycle $\langle\mathbb
R,\varphi,(Y,\mathbb Z_{+},\sigma)\rangle$ generated by
(\ref{eqSDDE1}) (see Example \ref{exDE1}). Since the function $f$
is monotone then by Lemma \ref{l8.03} the cocycle $\varphi$
generated by the equation (\ref{eqSDDE1}) satisfies condition
(\ref{eqM3}) and, consequently, the cocycle $\varphi$ is
positively uniformly stable. Denote by $x_0:=(u_{0},f)\in
X=\mathbb R\times H^{+}(f)$. It easy to check that under the
conditions of Theorem the point $x_0$ of the skew-product
dynamical system $(X,\mathbb Z_{+},\pi)$ ($X:=\mathbb R\times
H(f)$ and $\pi :=(\varphi,\sigma)$) is positively Lagrange stable.
Thus $\omega_{x_0}$ is a nonempty, compact and invariant set of
$(X,\mathbb Z_{+},\pi)$. Let
\begin{equation}\label{eqNDSDE1}
\langle (X,\mathbb Z_{+},\pi),(Y,\mathbb Z_{+},\sigma),h\rangle
\end{equation}
be a nonautonomous dynamical system associated by the cocycle
$\varphi$ ($X:=\mathbb R\times H^{+}(f)$, $\pi :=(\varphi,\sigma)$
and $h:=pr_{2}:X\to Y:=H^{+}(f)$). Note that
$h(\omega_{x_0})=\omega_{f}$ and by Theorem \ref{thV2} it is well
defined the two-sided nonautonomous dynamical system $\langle
(\omega_{x_0},\mathbb Z,\pi),(\omega_{f},\mathbb
Z,\sigma),h\rangle$. Since the function $f$ is remotely almost
periodic and positively Lagrange stable then the set $\omega_{f}$
is a compact and minimal set consisting of almost periodic
functions $g\in \omega_{f}$. To finish the proof of Theorem it
suffices to apply Theorem \ref{thV6} to nonautonomous dynamical
system (\ref{eqNDSDE1}). Theorem is proved.
\end{proof}

Below we give an example which illustrates the results above.

\begin{example}\label{exDE2} (Beverton-Holt equation)
{\em The periodic Beverton-Holt equation
\begin{equation}\label{eqBH1}
x_{n+1}=\frac{\mu K_n x_n}{K_n +(\mu -1)x_n}
\end{equation}
$(K_{n+k}=K_n)$ has been studied by Cushing and Henson \cite{CH}
and Elaydi and Sacker \cite{ES}.

Below we will suppose that the following conditions hold:
\begin{enumerate}
\item[(C1)] the sequence $\{K_n\}_{n\in \mathbb Z}$ is remotely
almost periodic; \item[(C2)] $\alpha <\beta$ are two positive
constants such that $\alpha \le K_n \le \beta$ for all $n\in
\mathbb Z$; \item[(C3)] $\mu\beta^{2}\alpha^{-2}\le 1.$
\end{enumerate}

Let
$$
f(n,x):=\frac{\mu K_n x}{K_n +(\mu -1)x}
$$
for all $n\in \mathbb Z_{+}$ and $x\in \mathbb R_{+}$. It easy to
check that the function $f$ possesses the following properties:
\begin{enumerate}
\item $f(n,x)\ge 0$ for all $n\in \mathbb Z_{+}$ and $x\in \mathbb
R_{+};$ \item
$$
f'_{x}(n,x)=\frac{\mu K_n^{2} }{(K_n +(\mu
-1)x)^{2}}>0
$$
for every $n\in \mathbb Z_{+}$ and $x\in \mathbb R_{+};$ \item
$$
f''_{x^2}(n,x)=-\frac{2\mu (\mu -1)K_n^{2} }{(K_n +(\mu
-1)x)^{3}}<0
$$
for all $n\in \mathbb Z_{+}$ and $x\in \mathbb R_{+}.$
\end{enumerate}

\begin{lemma}\label{lBH1} Assume that $0 <\alpha<\beta $ and $\mu \beta^{2}\alpha^{-2}\le
1$. Then the function
\begin{equation}\label{eqBH_1}
f(n,x):=\frac{\mu K_n x}{K_n +(\mu -1)x}\nonumber
\end{equation}
satisfies the inequality
\begin{equation}\label{eqBH2.1}
|f(t,x_1)-f(t,x_2)|\le |x_1 -x_2|\nonumber
\end{equation}
for all $t\in \mathbb Z_{+}$ and $x_1,x_2\in \mathbb R_{+}$.
\end{lemma}
\begin{proof} To prove this statement we note that
\begin{eqnarray}\label{eqBH2}
& |f(t,x_1)-f(t,x_2)|=|\frac{\mu K_n x_1}{K_n +(\mu -1)x_1} -
\frac{\mu K_n x_2}{K_n +(\mu -1)x_2}|= \nonumber \\
& \frac{\mu K(t) |x_1-x_2|}{(K(t)+(\mu -1)x_1)(K(t)+(\mu
-1)x_2)}\le \frac{\mu \beta^{2}|x_1-x_2|}{\alpha^{2}}\le
|x_1-x_2|\nonumber
\end{eqnarray}
for every $t\in \mathbb Z_{+}$ and $x_1,x_2\in \mathbb R_{+}$.
Lemma is proved.
\end{proof}

\begin{lemma}\label{l8.3}\cite{CM_2006} Let
$$
f(n,x):=\frac{\mu K_n x}{K_n +(\mu -1)x}
$$
for all $n\in \mathbb Z_{+}$ and $x\in \mathbb R_{+}$. Then the
following statements hold:
\begin{enumerate}
\item
$$
f(n,x)-x=\frac{(\mu -1)x(K_n-x)}{K_n +(\mu -1)x}
$$
for every $n\in \mathbb Z_{+}$ and $x\in \mathbb R_{+};$ \item
$f(n,K_n)=K_n$ for all $n\in \mathbb Z;$ \item
$$
f(n,x)-\frac{\mu}{\mu -1}K_n=-\frac{\mu K_n^2}{(K_n +(\mu
-1)x)(\mu -1)}<0
$$
for every $n\in \mathbb Z_{+}$ and $x\in \mathbb R_{+}.$
\end{enumerate}
\end{lemma}

\begin{coro}\label{cor8.1} Let
$$
f(n,x):=\frac{\mu K_n x}{K_n +(\mu -1)x}
$$
for all $n\in \mathbb Z_{+}$ and $x\in \mathbb R_{+},$ then
\begin{equation}\label{eq8.3}
\limsup\limits_{n\to +\infty}\vert \varphi(n,u,f)\vert \le
\frac{\mu}{\mu -1}\beta \nonumber
\end{equation}
for every $u\in \mathbb R_{+}$.
\end{coro}

\begin{theorem}\label{thBH1} Assume that the following conditions
are fulfilled:
\begin{enumerate}
\item the function $K\in C(\mathbb Z_{+},\mathbb R)$ is remotely
stationary (respectively, remotely $\tau$-periodic or remotely
almost periodic); \item there are $0<\alpha <\beta$ such that
$\alpha \le K(t)\le \beta$ for every $t\in \mathbb Z_{+}$; \item
$\mu \beta^{2}\alpha^{-2}\le 1$.
\end{enumerate}

Then every solution $\varphi(t,u,f)$ of the Beverton-Holt equation
(\ref{eqBH1}) is remotely stationary (respectively, remotely
$\tau$-periodic or remotely almost periodic).
\end{theorem}
\begin{proof}
To prove this statement, taking into account Lemma \ref{lBH1} and
Corollary \ref{cor8.1}, it suffices to apply Theorem
\ref{thDEM3.1} to Beverton-Holt equation (\ref{eqBH1}).
\end{proof}
}
\end{example}

\section{Funding}

This research was supported by the State Programs of the Republic
of Moldova "Monotone Nonautonomous Dynamical Systems
(24.80012.5007.20SE)", "Remotely Almost Periodic Solutions of
Differential Equations (25.80012.5007.77SE)" and partially was
supported by the Institutional Research Program 011303 "SATGED",
Moldova State University.

\section{Data availability}

No data was used for the research described in the article.

\section{Conflict of Interest}

The author declares that he does not have conflict of interest.

\medskip
\textbf{ORCID:} https://orcid.org/0000-0002-2309-3823

\end{document}